\documentclass[12pt]{amsart}
\usepackage{amsfonts}

\usepackage{amscd}
\usepackage{amsmath}
\usepackage{amssymb}
\usepackage{amsthm}
\usepackage{epsf}
\usepackage{latexsym}
\usepackage{verbatim}
\usepackage{color}
\input epsf.tex
\usepackage{bbm}

\usepackage{epstopdf}
\usepackage[pdftex]{graphicx}

\newtheorem{theorem}{Theorem}[section]
\newtheorem{lemma}[theorem]{Lemma}

\newtheorem{proposition}[theorem]{Proposition}

\theoremstyle{definition}

\newtheorem{example}[theorem]{Example}
\newtheorem{problem}[theorem]{Problem}

\numberwithin{figure}{section}
\numberwithin{equation}{subsection}

\newcommand\D{{\mathbb D}}
\newcommand\C{{\mathbb C}}
\newcommand\Chat { {\hat{\C}} }

\renewcommand\P{{\mathbb P}}
\newcommand\R{{\mathbb R}}

\newcommand\del{\partial}
\newcommand\delbar{\bar{\del}}

\newcommand\eps{\varepsilon}

\begin{document} 
\title{Convex shapes and harmonic caps}

\author{Laura DeMarco}
\address{Department of Mathematics\\
         Northwestern University\\
         Evanston, IL 60208 \\ 
         U.S.A. }
\email{demarco@math.northwestern.edu }

\author{Kathryn Lindsey}
\address{Department of Mathematics\\
         University of Chicago\\
         Chicago, IL 60637 \\ 
         U.S.A.}        
\email{klindsey@math.uchicago.edu}

\begin{abstract} 
Any planar shape $P\subset \C$ can be embedded isometrically as part of the boundary surface $S$ of a convex subset of $\mathbb{R}^3$ such that $\partial P$ supports the positive curvature of $S$.  The complement $Q = S \setminus P$ is the associated \emph{cap}.  We study the cap construction when the curvature is harmonic measure on the boundary of $(\Chat\setminus P, \infty)$.  Of particular interest is the case when $P$ is a filled polynomial Julia set and the curvature is proportional to the measure of maximal entropy.  
\end{abstract}

\date{\today}

\maketitle

%%%%%%%

\section{Introduction}

A {\bf planar shape} is a compact, connected subset of the Euclidean plane that contains at least two points and has connected complement.  Given a probability measure $\mu$ supported on the boundary of a planar shape $P$, we investigate the existence of a conformal metric $\rho = \rho(z) |dz|$ on the Riemann sphere $\Chat$ so that 
\begin{enumerate}
\item	$P$, with its Euclidean metric from $\R^2$, embeds locally-isometrically into $(\Chat, \rho)$; and 
\item the curvature distribution $\omega_\rho = -\Delta \log \rho(z)$ on $\Chat$ is equal to the push-forward of $4\pi \mu$ under the embedding.
\end{enumerate}
%extending $(P, |dz|)$ such that its curvature distribution is equal to $4\pi \mu$.  More precisely, we ask for a conformal metric $\rho(P,\mu) = \rho(z) |dz|$ on $\Chat$ so that $P$ isometrically embeds into $(\Chat, \rho(P,\mu))$ and so that the curvature distribution 
%	$$\omega_{\rho(P,\mu)} = -\Delta \log \rho (z)$$
%is equal to the push-forward of $4\pi \mu$ under the embedding.  
If $\rho$ exists, then it is uniquely determined up to isometry (c.f.~\cite[\S3.5, Theorem 1]{Alexandrov:polyhedra}), and we will denote it by $\rho(P,\mu)$.  

A.~D.~Alexandrov's theorems on convex surfaces \cite{Alexandrov:1942, Alexandrov:intrinsic, Alexandrov:polyhedra} assert that any abstract metrized sphere with non-negative curvature is isometric to the  boundary surface  of a convex body in $\mathbb{R}^3$ with its induced metric (unique up to rigid motions of $\mathbb{R}^3$).   In particular, the metrized sphere $(\Chat, \rho(P,\mu))$ will have a unique convex 3D realization. The convex body may be {\bf degenerate}, meaning that it lies in a plane and the sphere is viewed as the double of a convex planar region.  Conversely, the surface of any compact, convex body in $\R^3$ (not contained in a line) may be endowed with a complex structure and uniformized so that it is isometric to the Riemann sphere with a conformal metric of non-negative curvature; see, e.g. \cite{Reshetnyak}.  Thus, the existence of $\rho(P,\mu)$ may be viewed as a problem of ``folding" the shape $P$ into $\R^3$ and taking its convex hull, in such a way that the curvature of the resulting convex body is given by $4\pi \mu$.  

The complement of $P$ in $(\Chat, \rho(P,\mu))$ will be called the {\bf cap} of $(P, \mu)$ and denoted by $\hat{P}_\mu$.  By construction, the metric on the cap is flat, so there is a locally isometric development map
	$$D:  \hat{P}_\mu \to (\C, |dz|).$$
We say the cap is {\bf planar} if the development $D$ is injective.  

Our first observation is that there always exists a probability measure $\mu$ supported on $\del P$ so that the metric $\rho(P,\mu)$ exists (see \S\ref{naive} for a simple but degenerate construction).  We also observe that not all caps are planar, and we give examples in Section \ref{s:spirals}.  

\bigskip
\noindent
{\bf The harmonic cap.}  
We are especially interested in the case where $P$ is a connected filled Julia set $K(f)$ of a polynomial $f: \C\to\C$ and the prescribed measure $\mu$ is the measure of maximal entropy supported on the boundary of $K(f)$; see details in \S\ref{polynomials}.  This metrized sphere was defined in \cite[Section 12]{D:lyap} for an arbitrary rational map $f: \P^1 \to \P^1$ of degree $>1$.  Questions about the features of its 3-dimensional realization were first posed by C.~McMullen and W.~Thurston.  

To this end, we examine arbitrary planar shapes $P\subset \C$, and we let $\mu$ be the harmonic measure for the domain $\Chat\setminus P$ relative to $\infty$.  By definition, $\mu$ is the push-forward of the Lebesgue measure on the unit circle $S^1$ (normalized to have total mass 1) under a conformal isomorphism $\Phi : \C\setminus \overline{\D} \to \C\setminus P$; the measure $\mu$ is well defined even if $\Phi$ is not everywhere defined on $S^1$.   In this setting, the metric $\rho(P,\mu)$ is simply an extension of the Euclidean metric $|dz|$ on $P$; it can be expressed in terms of the Green function 
	$$G_P(z) = \log|\Phi^{-1}(z)|$$
for $z \in \C\setminus P$.  Setting $G_P(z) = 0$ for $z\in P$, we have 
	$$\rho(P,\mu) = e^{-2G_P(z)}|dz|.$$
Observe that the metric $\rho(P,\mu)$ is continuous on all of $\Chat$:  $G_P$ is continuous on $\C$ (by solvability of the Dirichlet problem on simply-connected domains), and it grows as $\gamma + \log|z| + o(1)$ as $z\to\infty$ for some $\gamma\in\R$.  The cap $\hat{P}_\mu$ is called the {\bf harmonic cap} of $P$.  

\begin{theorem} \label{t:harmonic}
Let $P$ be any planar shape and let $\mu$ be the harmonic measure on $\del P$, relative to $\infty$.   Let $\Phi : \Chat\setminus \overline{\D} \to \Chat\setminus P$ be a conformal isomorphism with $\Phi(\infty) = \infty$.  A Euclidean development of the harmonic cap $\hat{P}_\mu$ is given by the locally univalent function $g: \D\to\C$ defined by 
	$$g(z) = \int_0^z \Phi'(1/x) \, dx.$$
Moreover, there exist planar shapes $P$ for which the harmonic cap is not planar.  
\end{theorem}

\noindent
As an example, the harmonic cap of a closed interval is planar; its development is shown in Figure \ref{f:Chebcap} for $P=[-2,2]$ where $g(z) = z - z^3/3$.  A non-planar example is described in \S\ref{non-planar}.

Theorem \ref{t:harmonic} allows one to appeal to the theory of univalent functions for conditions on $P$ that guarantee planarity of the harmonic cap.  If the harmonic cap is planar, then the construction can be iterated, to find the harmonic cap of the development of a harmonic cap.  It would be interesting to understand the properties of this dynamical system on a class of planar shapes.  (The closed unit disk is a fixed point of this operation; see Example \ref{disk}.)

\bigskip
\noindent
{\bf Constructing a cap.}  
Given the data of a conformal metric $(\Chat, \rho)$ with non-negative curvature distribution, it is a notoriously difficult problem to construct the 3D realization, even for polyhedral metrics (as we discuss below).  But it turns out that a development of a cap $\hat{P}_\mu$ in $\C$ can be easily produced on the computer.  

For planar shapes that are Jordan domains with rectifiable boundaries, a cap $\hat{P}_\mu$ will have boundary of the same length as $\del P$.  A {\bf perimeter gluing} of $P$ and $\hat{P}_\mu$ is the boundary identification (by arclength) between $\del P$ and $\del \hat{P}_\mu$ that produces $(\Chat, \rho(P,\mu))$.  

\begin{theorem}  \label{t:parametrization}
Let $P$ be a planar shape with a piecewise-differentiable Jordan curve boundary, and let $\mu$ be a nonnegative Borel probability measure supported on the boundary of $P$. Let $s$ be a counterclockwise, unit-speed parametrization of $\partial P$, and write $s(t) = \int_0^t e^{i\alpha(x)} \, dx$ for a real-valued function $\alpha$.  If the cap $\hat{P}_\mu$ exists, then the boundary of its Euclidean development is parameterized in the clockwise direction by 
	$$\hat{s}(t) = \int_0^t e^{i (\alpha(x) -  \kappa(x))} \, dx$$
where $\kappa(t) = 4\pi \mu(s(0,t])$, and the perimeter gluing is given by $s(t) \sim \hat{s}(t)$, .  
\end{theorem}

Given an arbitrary planar shape $P$, we can approximate it by a shape $P'$ with piecewise-differentiable Jordan curve boundary and approximate any given measure $\mu$ on $\del P$ with a probability measure supported on the boundary of $P'$.  In this way, Theorem \ref{t:parametrization} supplies a straightforward strategy to illustrate the caps.  In practice, we use polygonal approximations to the planar shape $P$ with discrete curvature supported on the vertices.  See Figures \ref{square cap} and \ref{basilica cap}.  A theorem of Reshetnyak states that weak convergence of the curvature distributions as measures on $\Chat$ implies convergence of the metrics \cite[Theorem 7.3.1]{Reshetnyak}, \cite{isothermal}.  

\begin{figure} [h]
\includegraphics[width=3in]{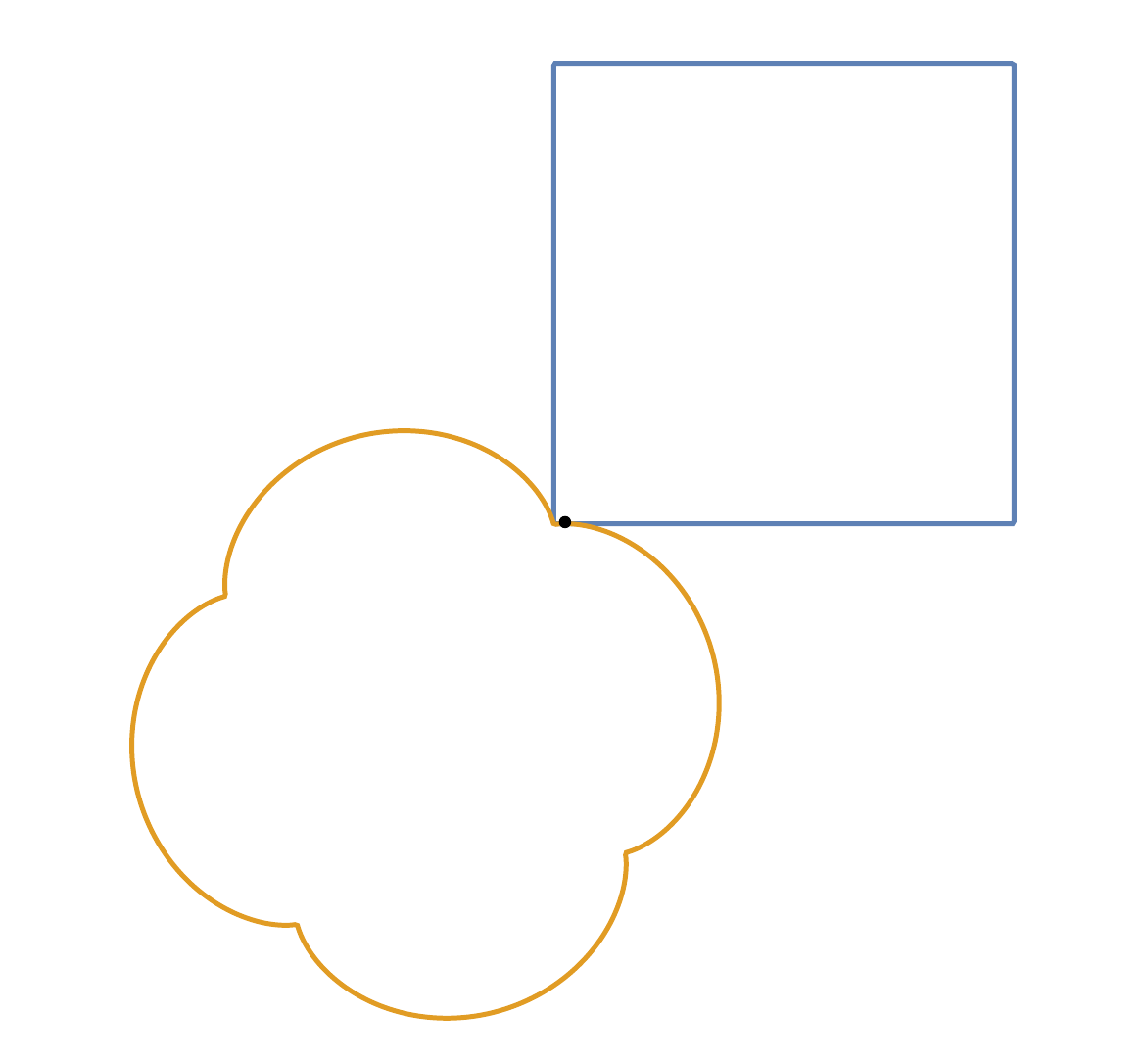} 
\caption{ \small In blue, a square; in orange, its harmonic cap, with one attaching point indicated in black.  The perimeter gluing is by arclength.  There is a unique realization of the glued shapes as the boundary of a convex body in $\R^3$.  The harmonic measure on the boundary of the square was approximated by a discrete measure supported on 500 points, using the Riemann mapping function \cite{RiemannMapping} in Sage \cite{sage}.  Image generated with Mathematica.}
\label{square cap}
\end{figure}

\begin{figure} [h]
\includegraphics[width=4.5in]{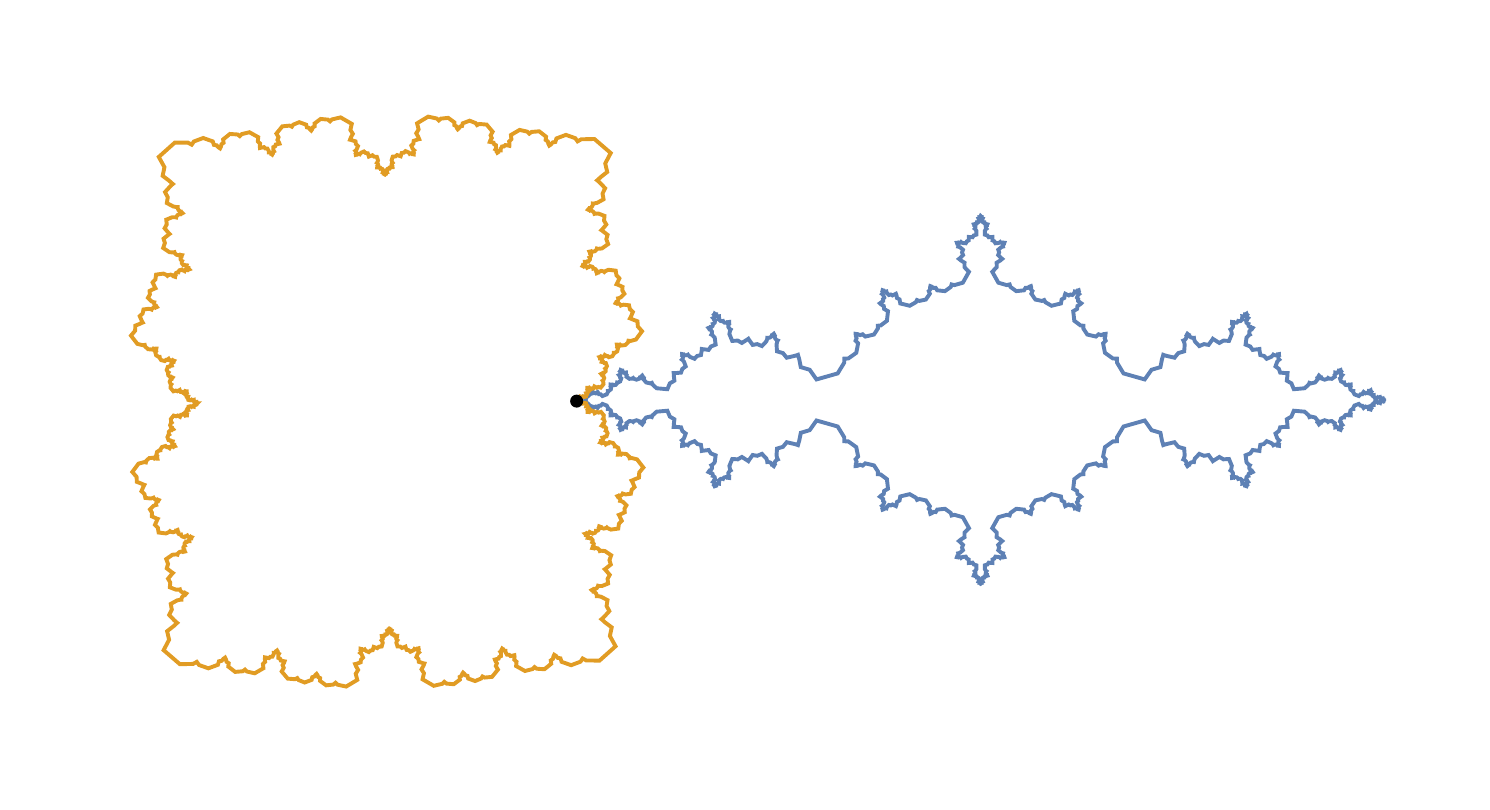} 
\caption{\small In blue, a polygonal approximation to the filled Julia set of $f(z)=z^2-1$ with $2^{11}$ vertices, the preimages of $z=2.0$ under $f^{11}$.  The discrete probability measure that assigns equal mass to each of its $2^{11}$ vertices approximates the harmonic measure on the filled Julia set.  In orange, the polygonal cap associated to this polygon with discrete curvature measure.  There is a unique realization of the glued shapes as the boundary of a convex body in $\R^3$.  Image generated with Mathematica.}
\label{basilica cap}
\end{figure}

For polygonal planar shapes with arbitrary probability measures $\mu$ supported on their vertices, our cap-drawing algorithm (which follows the proof of Theorem \ref{t:parametrization}) can be used to draw the parametrization $\hat{s}$, independent of the existence of the metric extension $\rho(P,\mu)$.  For many examples, the curve $\hat{s}$ fails to form a closed loop or has a shape that cannot be the boundary parametrization of any Euclidean development of a cap (e.g., it may have positive winding number around a point in the plane, while the boundary of a cap development, traversed in the clockwise direction, will wind non-positively around all points).  For example, if $P$ is a triangle, there is a unique measure $\mu$ supported on the vertices of $P$ that gives rise to a cap: any associated cap is necessarily a triangle whose sidelengths are the same as those of $P$, implying the cap is a reflected copy of $P$, the convex shape is degenerate, and $\mu(v) = (\pi - \theta)/(2\pi)$ where $v$ is a vertex of $P$ with internal angle $\theta$.  In general, the questions of when the metric $\rho(P,\mu)$ exists and when the cap $\hat{P}_\mu$ is planar are quite delicate, even in the polygonal setting.  

\begin{problem}  \label{polygon problem}
For polygons of $N$ sides, with side lengths $\{\ell_1, \ldots, \ell_N\}$ and internal angle $\theta_i$ at each of its vertices $v_i$, give an explicit description of the discrete curvature distributions $\mu = \{\mu_i\}$ supported on the vertices $v_i$ so that the metric $\rho(P,\mu)$ exists.  Provide conditions under which the polygonal cap $\hat{P}_\mu$ is planar.  
\end{problem}

Problem \ref{polygon problem} is related to the geometry of the space of polygons with fixed side lengths and no boundary crossings, which, to our knowledge, has never been described.  See \cite{CDR} where it is proved that the space is connected and contractible.

\bigskip
\noindent
{\bf The 3-dimensional realization.}  
Recall, by Alexandrov's theorems (\cite{Alexandrov:polyhedra, Alexandrov:intrinsic, Alexandrov:1942}, \cite{Pogorelov}),  for nonnegative $\mu$ {\em there is a unique way to fold the Euclidean development} of $P$ and $\hat{P}_\mu$ to form the boundary surface of a convex shape in $\R^3$.  We may view the output of the cap-drawing algorithm, as in Figures \ref{square cap} and \ref{basilica cap}, as paper cut-outs to be creased and glued to form the desired shape.  Unfortunately, the exact shape of the 3-dimensional realization is not at all clear from the development alone.  Even the set of folding lines inside $P$ and $\hat{P}_\mu$ is a mystery in general.   Quoting from Alexandrov in translation \cite[p.100]{Alexandrov:polyhedra}, ``To determine the structure of a polyhedron from a development, i.e., to indicate its genuine edges in the development, is a problem whose general solution seems hopeless."  But in the case of harmonic measure on a planar shape, especially when the shape is the filled Julia set of a polynomial, there may be specialized ways to attack the problem.

Not long ago, Bobenko and Izmestiev devised an illuminating and constructive proof of Alexandrov's realization theorem for polyhedral metrics \cite{Bobenko:Izmestiev}, implementing their algorithm and making it publicly available.  Unfortunately, the algorithm was not practical for the polyhedra that closely approximate the metrics for polynomial Julia sets \cite{Bartholdi:pc}.  Laurent Bartholdi modified their strategy to handle some dynamical examples, such as the filled Julia set of $f(z) = z^2-1$ shown in Figure \ref{3D basilica}.  
 
\begin{figure} [h]
\includegraphics[width=2.5in]{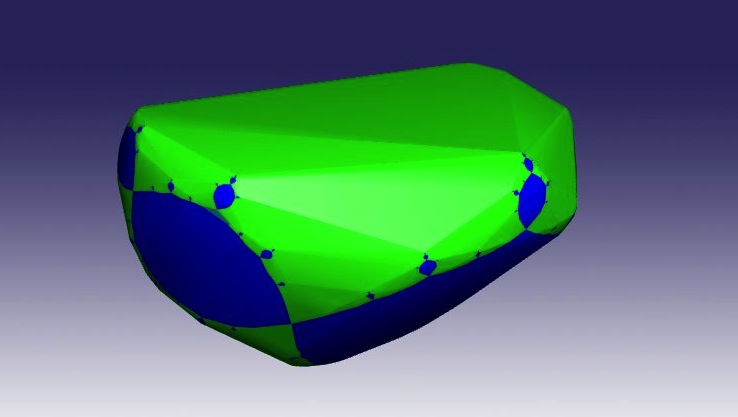}
\includegraphics[width=2.41in]{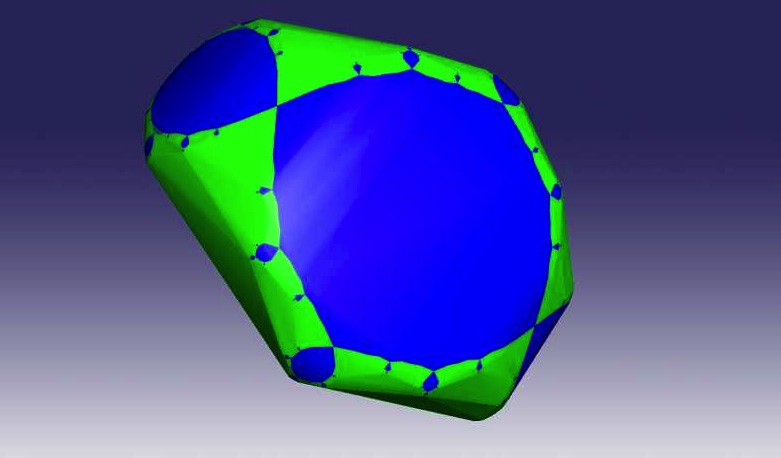} 
\caption{\small Two views of Bartholdi's polyhedral approximation to the 3D realization of the filled Julia set of $f(z) = z^2 - 1$ with its harmonic measure, computed with $2^{11}$ vertices.  An illustration of the filled Julia set is superimposed onto the image.  Graphic created with \texttt{glc player}.}
\label{3D basilica}
\end{figure}

Formally, the convex 3D realization of $(\Chat, \rho(P, \mu))$ determines a Euclidean lamination on the interiors of $P$ and $\hat{P}_\mu$, consisting of the geodesic line segments that must be folded to form the 3D shape.  We call this the {\bf bending lamination} of the pair $(P, \mu)$.  If one also retains the data of the dihedral angles (the amount of the fold along each leaf of the lamination), we obtain a measured lamination, uniquely determined by the pair $(P, \mu)$.  We leave the following as an open problem: 
 
\begin{problem}  \label{p:realization}
Suppose $\mu$ is the harmonic measure relative to $\infty$ on the boundary of a planar shape $P$. Describe the (measured) bending lamination of $(P,\mu)$.  
\end{problem}

\bigskip\noindent
{\bf Other comments and acknowledgments.}  In the course of this project, we were introduced to the vast literature of the computational geometry community.  Quite a bit of research has gone towards visualizing the 3D realizations of Alexandrov's convex polyhedral metrics and related problems. Most notably, we mention that we learned much from the work of Demaine and O'Rourke and their co-authors; see, e.g. \cite{DemaineORourkeSurvey, Demaine:O'Rourke:book}.   

We would like to thank Curt McMullen and, posthumously, Bill Thurston, for introducing us to this problem and for many interesting conversations on the topic over the past 15 years.  In particular, the idea of representing a Julia set and its cap as paper cut-outs is due to Thurston.  Our perspective on caps and bending is also inspired by the theory of pleated surfaces and Thurston's study of spaces of polyhedra \cite{Series, ThurstonShapesofPolyhedra},  and the geometry of filled Julia sets for homogeneous polynomial maps \cite{Hubbard:Papadopol}.  We are grateful to Laurent Bartholdi, Ilia Binder, Robert Connelly, David Dumas, and Amie Wilkinson for helpful discussions.   Finally, we thank the anonymous referee for many thoughtful, useful suggestions.  

Our research was supported by the National Science Foundation and the Simons Foundation.

%%%%%%%%%%%%%%%
\bigskip
\section{Caps, spirals, and Julia sets}
\label{s:spirals}

In this section, we observe that for every planar shape $P$, there is a probability measure $\mu$ on its boundary so that the metric $\rho(P,\mu)$ on $\Chat$ exists, by simple constructions in $\R^2$.  We provide examples to illustrate the failure of planarity of a cap.  We conclude the section with examples of harmonic caps coming from polynomial dynamical systems $f: \C\to \C$.  Formal definitions and the proofs of our theorems will be given in Sections \ref{s:metrics} and \ref{s:complex}.

\subsection{The naive cap}  \label{naive}
Let $P$ be a planar shape that is not contained in a line.  Let $\bar{P}$ be the convex hull of $P$ in the plane.  The {\bf naive cap} $\hat{P}$ is the union of $\bar{P}$ and a copy of each connected component of $\bar{P} \setminus P$ (the {\bf flaps}), glued along their boundaries in $\del \bar{P}$.  Then $P$ and $\hat{P}$ glue to determine a degenerate convex body, and the metrized sphere is a doubled copy of $\bar{P}$.  Its curvature is supported in the intersection of $\del P$ with $\del \bar{P}$.  Unfolding the flaps of the naive cap $\hat{P}$ determines a Euclidean development.  We can appeal to the Uniformization Theorem or to Reshetnyak's theorem on isothermal coordinates \cite[Theorem 7.1.2]{Reshetnyak} to conclude that this degenerate surface can be represented as a conformal metric on the Riemann sphere $\Chat$.  

If $P$ is an interval, then we can produce a cap by bending $P$  into an L-shape in the plane, introducing an angle at the midpoint of $P$, and then taking the convex hull of this new shape in $\R^2$.  Viewing the resulting triangle as a degenerate convex body in $\R^3$, we produce a metrized sphere with 3 concentrated points of curvature, at the two endpoints of $P$ and at its midpoint.  As $P = \del P$ in this example, we have shown the existence of a probability measure $\mu$ supported in $\del P$ and giving rise to a metric $\rho(P,\mu)$ on $\Chat$.  The developed cap $\hat{P}_\mu$ will be a rhombus.  For example, if the angle is chosen to be $\pi/3$, then the triangle will be equilateral, and $\mu$ will assign equal mass to each of the three cone points.

\subsection{The naive cap is not always planar} \label{non-planar naive}
Start with a convex polygonal shape in the plane with an external angle of about $\pi/16$ at one vertex.   Remove two very thin spiral channels from the polygon that begin on adjacent edges of the polygon and spiral around one another, as in the left image of Figure \ref{spiral1}.   If the spirals are sufficiently intertwined, then the spiral flaps on the developed naive cap will overlap.  The right side of Figure \ref{spiral1} shows the spirals reflected across the edges of the polygon.  

\begin{figure} [h]
\includegraphics[width=2.5in]{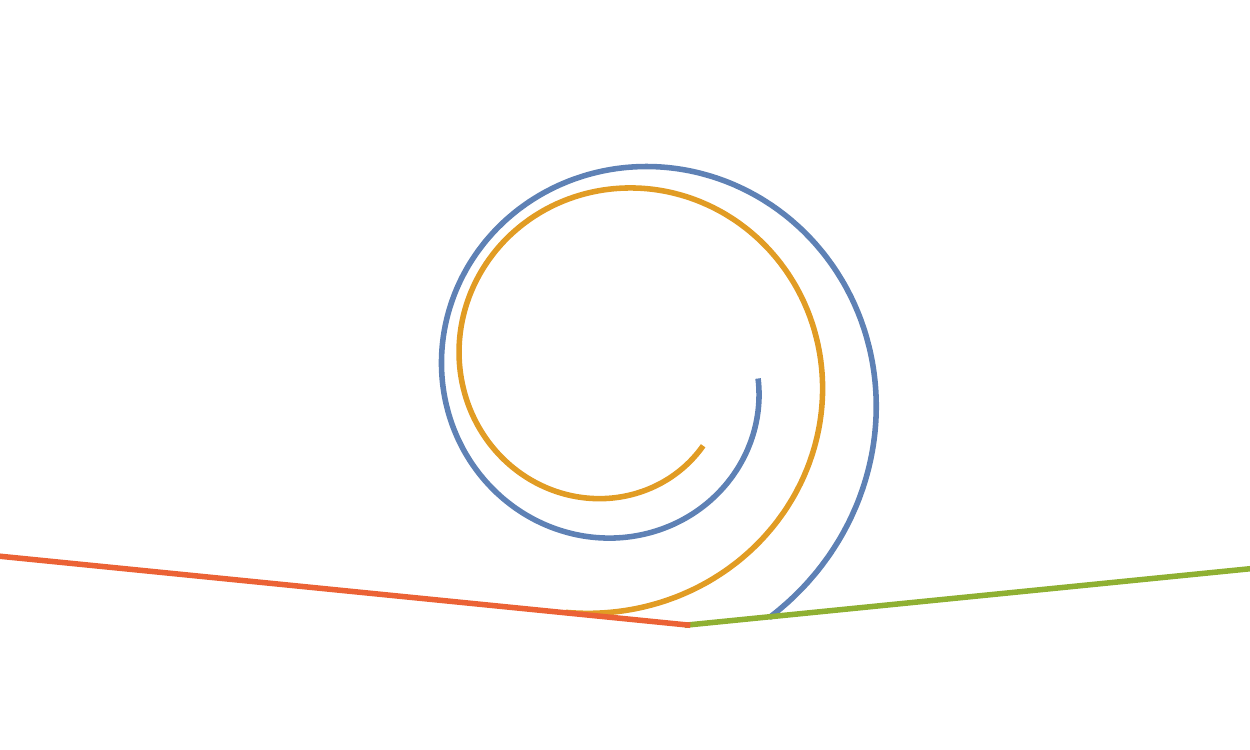}
\includegraphics[width=2.5in]{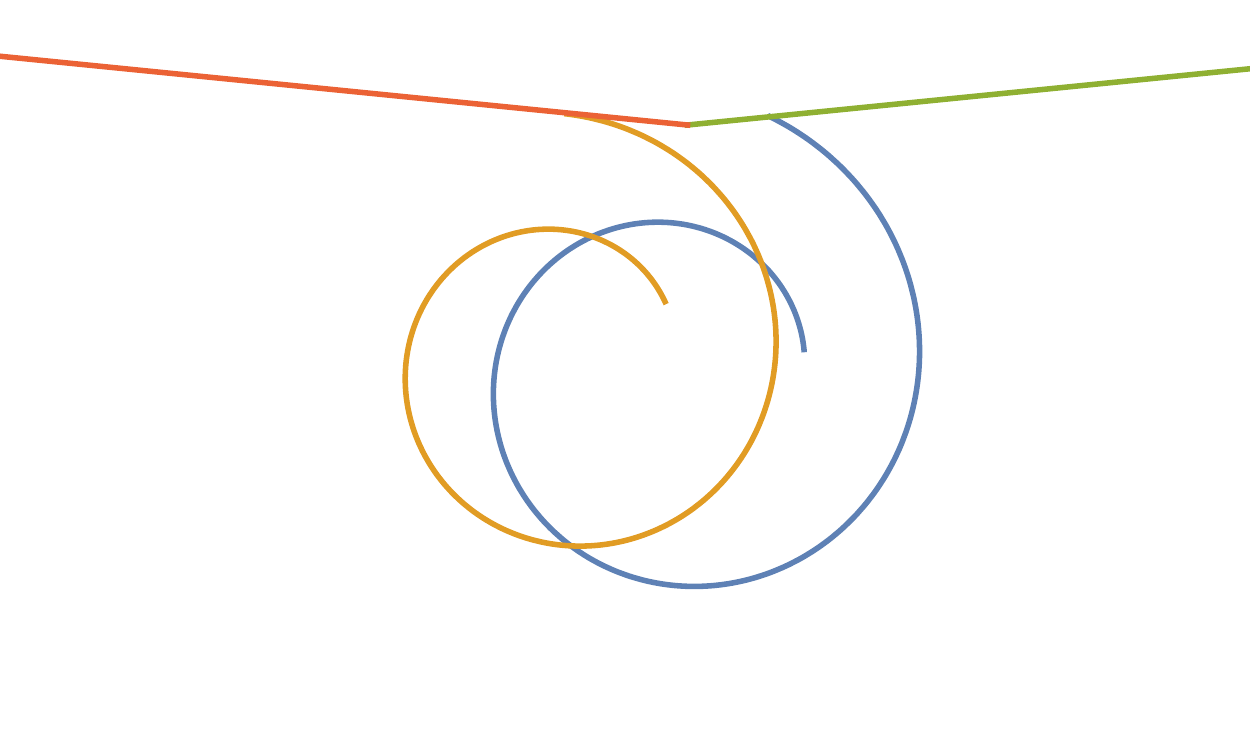}
\caption{ \small Left:  A piece of a convex polygon (lying above the red and green line segments) minus two narrow spiral channels (shown in orange and blue) that begin from adjacent edges of the polygon.  Each channel cut from the polygon is so narrow that we depict it as a curve.  Right: A piece of its naive cap (again, above the red and green segments) with the two spiral flaps reflected outward, illustrating a non-planar Euclidean development.}
\label{spiral1}
\end{figure}

\subsection{Non-planar example for harmonic measure}  \label{non-planar}
For the harmonic cap, it is possible to construct an example similar to that of \S\ref{non-planar naive}.  Indeed, very skinny channels removed from any planar shape will have negligible harmonic measure, and so we can arrange for overlapping spirals in the cap.

\begin{figure} [h]
\includegraphics[width=1.5in]{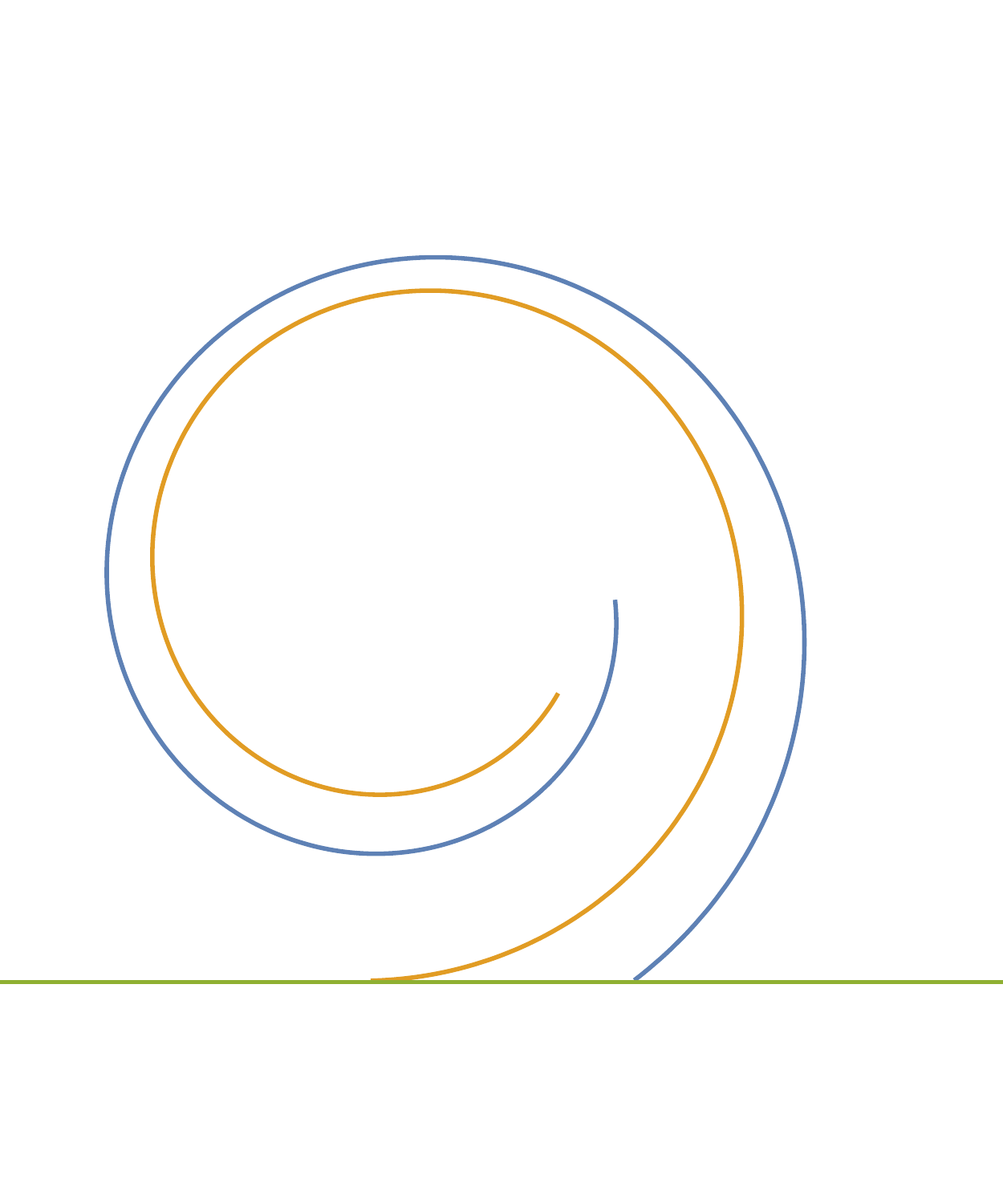} \qquad\qquad
\includegraphics[width=1.41in]{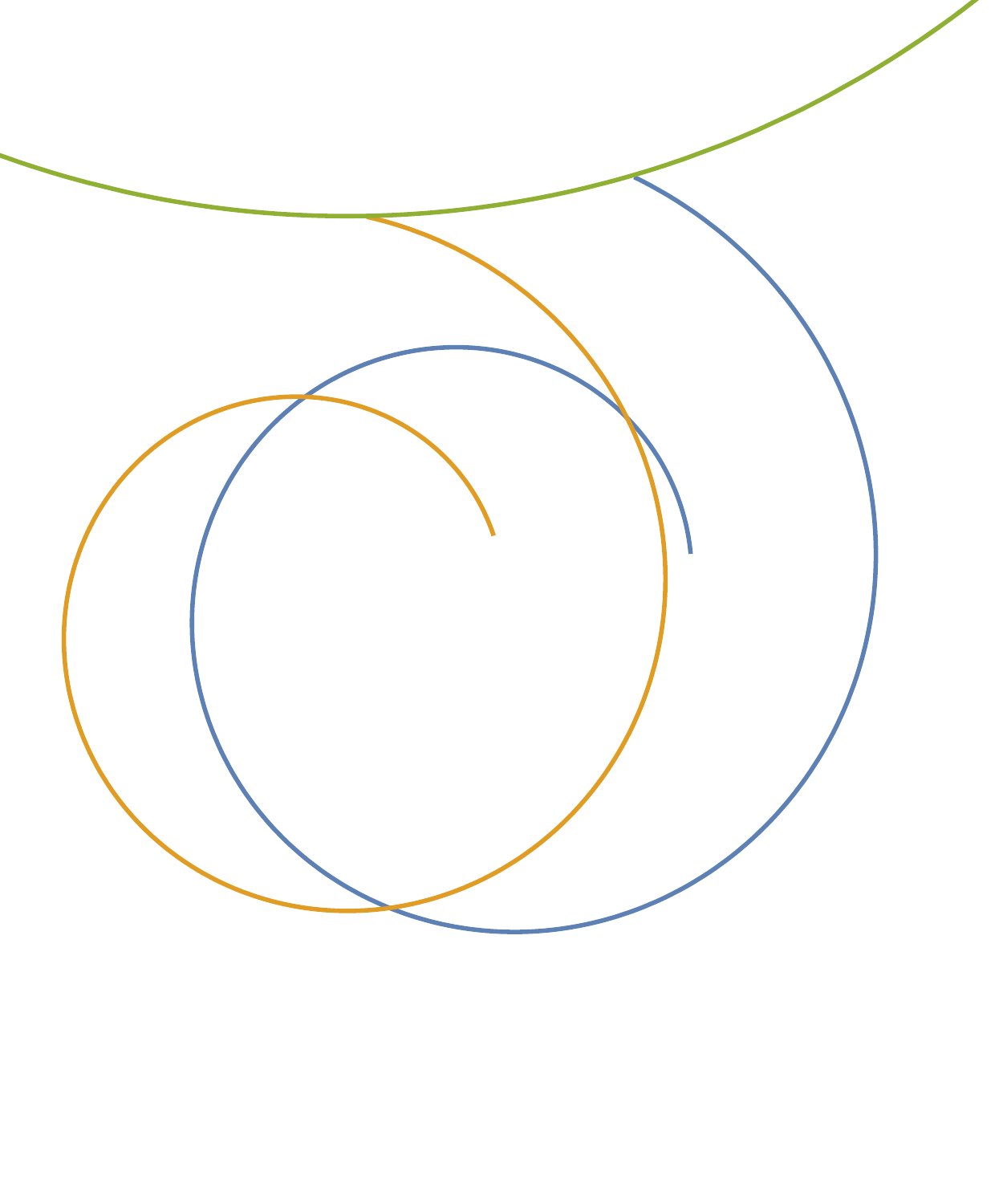}
\caption{ \small Left:  Two narrow spiral channels (shown in orange and blue) cut from the interior of a square planar shape (a segment of which is shown in green).  Each channel cut from the polygon is so narrow that we depict it as a curve.  Right:  The two spirals on the exterior of the clover-shaped harmonic cap of the square, illustrating a non-planar Euclidean development.  A complete and accurate picture of the harmonic cap of the square is shown in Figure \ref{square cap}.}
\label{spiral2}
\end{figure}

More precisely, begin with a square planar shape and choose a tiny $\eps>0$.   The harmonic cap for the square is shown in Figure \ref{square cap}.  Now remove two very skinny spiral channels from the square, emanating from a single edge, as in the left image of Figure \ref{spiral2}; the openings of each channel should have width smaller than $\eps$.  The openings of the two spiral channels can be placed at a specified distance apart from one another, so that the harmonic measure of the interval between them is approximately equal to $1/32$ of the total mass.  (The number $1/32$ is chosen because it is $1/4\pi$ times the curvature of $\pi/8$ for the polygon vertex shown in Figure \ref{spiral1}).  We can choose $\eps>0$ as small we wish so that the harmonic measure along the spiral boundaries is almost 0.  Indeed, as the width of the spiral channels shrinks to 0, the domains $\Chat\setminus P$ are converging in the Carath\'eodory sense to the complement of the square; see, e.g., \cite[\S3.1]{Duren:univalent}.  

Recall that the boundary of the cap development is parameterized by the formula of Theorem \ref{t:parametrization}.  The parametrization of the spirals on the cap, which will lie outside the clover-like harmonic cap for the square, will be essentially equal to a reflection of their original parametrizations (because $\kappa$ will be essentially constant along their boundaries, having chosen the harmonic measure of the spirals to be near 0).  On the other hand, the non-trivial portion of harmonic measure on the boundary of the square between the spiral-channel openings will curve the boundary of the cap so the spirals overlap.  The change in tangent direction of the clover cap between the two attaching points of the spirals will be $\pi/8$, by construction.  See Figure \ref{spiral2}.

\subsection{Polynomial Julia sets}  \label{polynomials}
Now assume that $f: \C\to\C$ is a complex polynomial of degree $d\geq 2$.  Its {\bf filled Julia set} is 
	$$K(f) = \{z \in \C: \sup_n |f^n(z)| < \infty\}.$$
Assume that $K(f)$ is connected, so it is a planar shape.  A planar development of its cap is given by the formula of Theorem \ref{t:harmonic}.  We can parameterize the boundary of the cap's development for smooth or polygonal approximations to $K(f)$ using Theorem \ref{t:parametrization}.  

The Green function for $K(f)$ can be computed dynamically, as 
	$$G_f(z) = \lim_{n\to\infty} \frac{1}{d^n}  \log^+|f^n(z)|.$$
The harmonic measure $\mu_f = \frac{1}{2\pi} \Delta G_f$ is the unique measure of maximal entropy for $f$, and its support is equal to the Julia set $J(f) = \del K(f)$ \cite{Brolin, Lyubich:entropy, FLM}.  The metric on $\Chat$ is defined by
	$$\rho_f = e^{-2G_f(z)}|dz|$$
for $z\in \C$, with curvature distribution $\omega_f = - \Delta \log \rho_f(z) = 4\pi \mu_f$.  

\begin{example} Let $f(z) = z^2$.  Then $K(f)$ is the closed unit disk and $G_f(z) = \log^+|z|$.  The measure $\mu_f$ is the Lebesgue measure on the circle.  By symmetry, the harmonic cap is also a closed disk of radius 1.    It follows that the convex realization in $\R^3$ is a degenerate closed disk.  
\end{example}

\begin{example} \label{Cheb} Let $f(z)= z^2-2$.  Then $K(f)$ is the real interval $[-2,2]$, and the metric on the sphere and the Euclidean development of the harmonic cap can be computed explicitly.  The Riemann map from the complement of the unit disk to the complement of $K(f)$ is given by 
	$$\Phi_f(z) = z+ \frac{1}{z}.$$
Applying Theorem \ref{t:harmonic}, the cap is the image of the holomorphic function $g: \D\to \C$ defined by
	$$g(z) = \int_0^z \Phi_f'(1/x) \, dx = \int_0^z (1 - z^2) \, dz = z - z^3/3.$$
See Figure \ref{f:Chebcap}.   The convex realization in $\R^3$ is degenerate.  
\end{example}

\begin{figure}[h]
\includegraphics[width=1.25in]{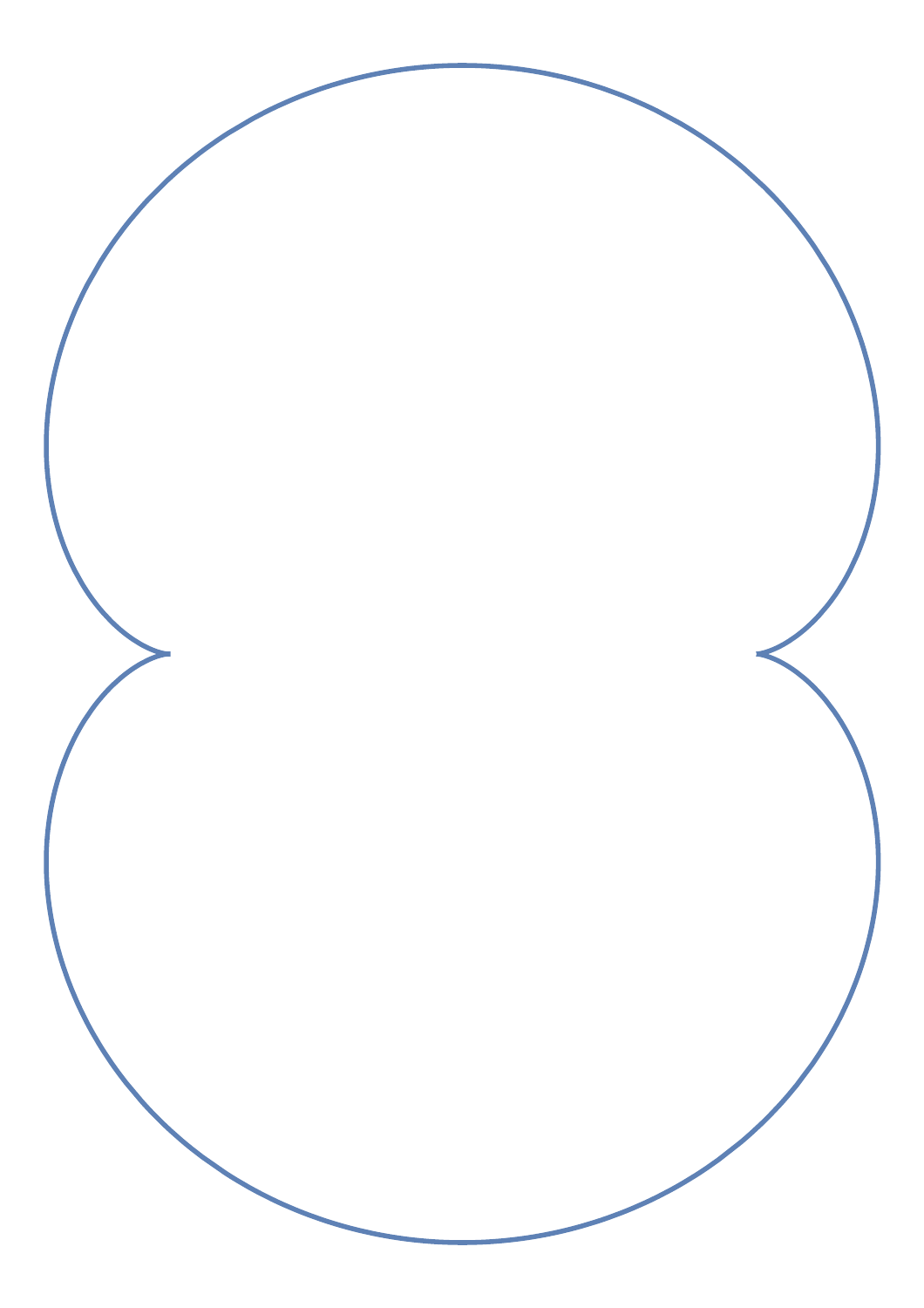} \qquad\qquad
\caption{ \small A Euclidean development of the cap for the real interval $P = [-2,2]$ equipped with its harmonic measure.  The figure shown is the image of the unit circle under $g(z) = z - z^3/3$, so the cusps lie at $z=\pm 2/3$.  To form the metrized sphere, the cap is folded in half along the segment joining the cusp points, and the interval $P$ forms the seam.  The resulting convex body is degenerate.    See Example \ref{Cheb}.}
\label{f:Chebcap}
\end{figure}

\begin{example} \label{caul}  Let $f(z) = z^2 + 1/4$.  A polygonal approximation to its filled Julia set and the harmonic cap are shown in Figure \ref{cauliflower cap}.  The convex realization in $\R^3$ is nondegenerate; indeed, if the filled Julia set were contained in a plane in $\R^3$, then its convex hull would also lie in the surface, and then the curvature could not be supported on all of $J(f)$.  
\end{example}

\begin{figure}[h]
\includegraphics[width=4.5in]{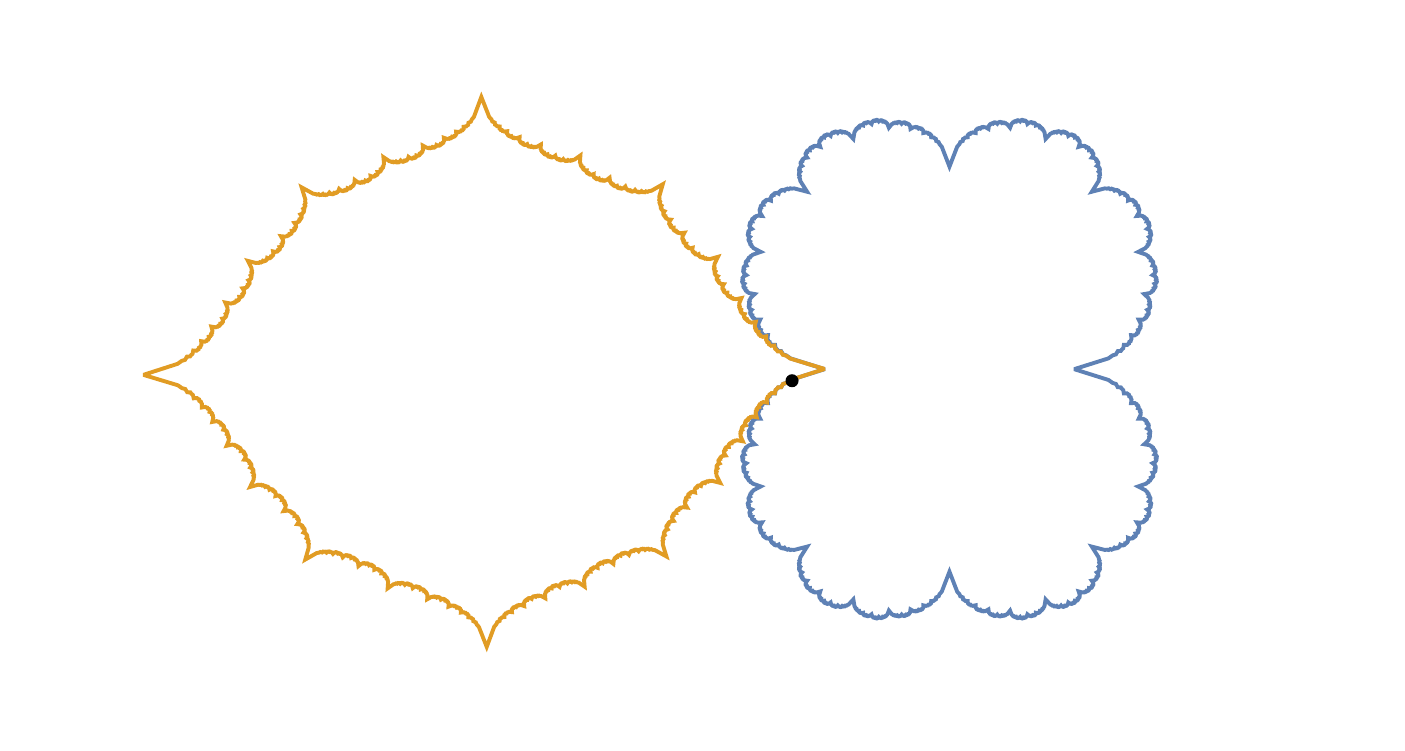} 
\caption{ \small In blue, a polygonal approximation to the filled Julia set of $f(z)=z^2+1/4$ with $2^{11}$ vertices, the preimages of $z=0.5$ under $f^{11}$.  The approximation to harmonic measure puts equal weight on each of the $2^{11}$ vertices.  In orange, the polygonal cap for this discrete curvature distribution.  A single attaching point is shown in black.  There is a unique realization of the glued shapes as the boundary of a convex body in $\R^3$.  Image generated with Mathematica.}
\label{cauliflower cap}
\end{figure}

%%%%%%%

\bigskip
\section{Metrics and curvature}
\label{s:metrics}

In this section, we formalize the notions of curvature and metric from the point of view of Euclidean geometry, and we prove Theorem \ref{t:parametrization}.  In Proposition \ref{p:smoothcurv}, we present an asymptotic formula for curvature when the boundary of the planar shape is a smooth Jordan curve, in terms of the circumference of small circles.  

\subsection{Polyhedra and cone angles}
A convex polyhedron in $\R^3$ is the intersection of finitely many closed halfspaces.  It is said to be degenerate if it lies in a plane.  When the polyhedron is non-degenerate and bounded, its boundary surface  is topologically a sphere, and the Euclidean metric from $\R^3$ induces an intrinsic path metric on the sphere.  If the polyhedron is degenerate and bounded, but not contained in a line, we will still view its boundary as a topological sphere, doubling the planar polygon and gluing along the polygonal boundary.

Abstractly, a convex polyhedral metric on a 2-dimensional sphere is an intrinsic metric with non-negative curvature concentrated at finitely many points.  In other words, in a small neighborhood of all but finitely many points, the surface is isometric to a region in $\R^2$.  In a neighborhood of each of the finitely many {\bf cone points}, the surface is isometric to the point of a cone.  The {\bf curvature} of a cone point is equal to the angle deficit at the point; that is, if the circumference of any small circle of radius $r$ centered at the cone point is equal to $C(r)$, then the curvature is equal to $(2\pi r - C(r))/r$.  By the Gauss-Bonnet formula, the sum of the curvatures over all cone points on the sphere is equal to $4\pi$.  

In \cite{Alexandrov:polyhedra}, A.~D.~Alexandrov examines the geometry of convex polyhedra in detail.  He presents his proof from \cite{Alexandrov:1942} that {\em any abstract polyhedral metric on a sphere is isometric to the boundary of a (possibly degenerate) convex polyhedron.} Furthermore, the polyhedron in $\R^3$ is unique, up to Euclidean isometries.  

Given a polyhedral metric on the sphere, and a simply-connected subset $U$ of the sphere minus its cone points, a {\bf Euclidean development} of $U$ is a local isometry $U \to \R^2$.  Suppose we are given the image $\mathcal{I}\subset \R^2$ of a Euclidean development of a full-area, simply-connected subset $U$ of the sphere.  Then, as a consequence of Alexandrov's theorem, the convex polyhedron in $\R^3$ is uniquely determined by $\mathcal{I}$ and the gluing along its boundary (that reconstructs the topological sphere).  In particular, the planar development and the gluing information will uniquely determine the edges of the polyhedron and their dihedral angles in $\R^3$ -- information that is not locally apparent.

\subsection{More general metrics of non-negative curvature} \label{curvature definition}
In \cite{Alexandrov:intrinsic}, Alexandrov presents the proof of a more general realization result; see Chapter 1 of \cite{Pogorelov} for a summary.  {\em Given any abstract intrinsic metric on the sphere of non-negative curvature, it is realizable as the boundary of a (possibly degenerate) convex body in $\R^3$.}  His argument relies on a convergence statement, first approximating the metric by polyhedral metrics, realizing the convex polyhedra, and then showing that the polyhedra converge to the desired convex body in $\R^3$.  

Curvature is carefully treated by Alexandrov.  It is defined by an additive set function $\omega$ as follows.  
%Pogorelov Chapter 1 page 9:  convergence of convex surfaces.  Any open set U that intersects the surface S must also intersect S_n for all n >>0.  (this sounds like the limit of S_n contains U.  Maybe convexity implies the reverse implication also?)
%P. Ch. 1, page 13:  convergence of convex surfaces implies uniform convergence of the distance function on points.
The curvature of a point is, as for a polyhedron, $2\pi$ minus the cone angle of the point.  That is, 
\begin{equation} \label{point curvature}
	\omega(\{x\}) = \lim_{r\to 0^+} \frac{2\pi r - C(x,r)}{r}
\end{equation}
where $C(x,r)$ is the circumference of the circle of radius $r$ centered at the point $x$.  The curvature of a geodesic line segment will always be 0.  The curvature of a (small) geodesic triangle is its internal angle surplus, defined as the sum of the internal angles of the triangle minus $\pi$.  The curvature of a more general region is computed by triangulation.  See \cite[Chapter 1, page 18]{Pogorelov}.  

Y.~G.~Reshetnyak, who was a student of Alexandrov, reformulated Alexandrov's theory of metrics and curvature on a surface in complex-analytic language, expressing curvature as a finite Borel measure \cite{Reshetnyak}.  We exploit this useful point of view in Section \ref{s:complex}.  

%%%%%%%%%%%%%%%%

\subsection{Parametrization of the cap}
Suppose that a planar shape $P$ is the closure of a Jordan domain with a piecewise-differentiable boundary.  Fix a nonnegative Borel  measure $\mu$ on the boundary of $P$. 
Let $L$ be the length of $\del P$.  Let $s$ be a piecewise-differentiable parametrization by arclength of the boundary of $P$, in the counterclockwise direction, and write
	$$s'(t) = e^{i \alpha(t)}$$
for a piecewise-continuous function $\alpha: [0,L] \to \R$.  For $t \in [0, L]$, we define a curvature function $\kappa: [0,L] \to [0, 4\pi]$ by $\kappa(0) = 0$ and 
\begin{equation} \label{curvaturefunction}
	\kappa(t) = 4\pi \mu(s(0,t])
\end{equation}
for all $t\in (0,L]$, so that $\kappa$ is monotone increasing with $\kappa(L) = 4\pi$.  Recall that Theorem \ref{t:parametrization} asserts that, if the cap $\hat{P}_\mu$ exists, then its boundary can be parameterized in the clockwise direction by 
	$$\hat{s}(t) = \int_0^t e^{i (\alpha(x) - \kappa(x))} \, dx.$$

\proof[Proof of Theorem \ref{t:parametrization}]
Suppose first that $P$ is a polygon in the complex plane and $\mu$ is a discrete probability measure supported on the vertices of $P$.  Denote the vertices of $P$ by $v_0, v_1, \ldots, v_N = v_0$, oriented counterclockwise, and set 
	$$\ell_j = |v_j - v_{j-1}|$$
to be the length of the $j$-th edge.  We may assume for simplicity that $v_0 = 0$ and $v_1 = \ell_1$ lies on the positive real axis.  Let $\theta_j$ be the internal angle of $P$ at vertex $v_j$, so that 
	$$\sum_{j=1}^N (\pi - \theta_j) = 2\pi$$
and
	$$\alpha(t) = \sum_{j=1}^{k-1} (\pi - \theta_j) \qquad \mbox{for } \sum_{j=1}^{k-1} \ell_j \leq t < \sum_{j=1}^k \ell_j$$
for each $k = 1, \ldots N$.  Thus $P$ is parameterized by 
	$$s(t) = \int_0^t e^{i\alpha(x)} \, dx.$$

If $\hat{P}_\mu$ exists, then it has a polygonal boundary with the same edge lengths as $P$.  We label its vertices in the clockwise direction by $\hat{v}_0, \hat{v}_1, \ldots, \hat{v}_N = \hat{v}_0$.  We may assume for simplicity that $\hat{v}_0 = v_0$ and $\hat{v}_1 = v_1$.  
The curvature condition implies that the internal angle $\hat{\theta}_j$ at vertex $\hat{v}_j$ must satisfy 
	$$4\pi \mu(v_j) = 2\pi - \theta_j - \hat{\theta}_j.$$
Therefore, the clockwise parametrization $\hat{s}$ of $\hat{P}_\mu$ will satisfy $\hat{s}'(t) = e^{i \hat{\alpha}(t)}$ with 
\begin{eqnarray*}
\hat{\alpha}(t)  
	&=& - \sum_{j=1}^{k-1} (\pi - \hat{\theta}_j) \qquad \qquad
			\mbox{for } \sum_{j=1}^{k-1} \ell_j \leq t < \sum_{j=1}^k \ell_j \\
	&=& \alpha(t) - \sum_{j=1}^{k-1} 4\pi \mu(v_j) \qquad 
			\mbox{for } \sum_{j=1}^{k-1} \ell_j \leq t < \sum_{j=1}^k \ell_j \\
	&=& \alpha(t) - \kappa(t)
\end{eqnarray*}
In other words, the parametrization of the boundary of $\hat{P}_\mu$ is given in a clockwise orientation by 
	$$\hat{s}(t) = \int_0^t e^{i (\alpha(x) - \kappa(x))} \, dx.$$

If $P$ is an arbitrary planar shape with piecewise-differentiable 
% this used to say piecewise smooth 
boundary, and if $\mu$ is any probability measure supported on the boundary of $P$, then the pair $(P, \mu)$ can be approximated by a sequence of polygons $(P_n, \mu_n)$ so that the vertices of $P_n$ lie in $\del P$ for all $n$, and $\mu_n$ is a discrete probability measure supported on the vertices of $P_n$.  We may construct the polygons $P_n$ so that the arclength parametrizations $s_n$ of $\del P_n$ converge uniformly to $s$ and that the angle functions $\rho_n \to \rho$ uniformly.  Furthermore, by choosing the vertices of $P_n$ carefully, we may assume that for every $\eps>0$, all atoms of mass at least  $\eps$ for $\mu$ are vertices of $P_n$ and atoms of $\mu_n$ for all $n\geq n(\eps) > 0$.  In this way, we can also arrange that the curvature functions $\kappa_n$ converge uniformly to the curvature function $\kappa$.  These choices for $(P_n, \mu_n)$ imply that the integrals
	$$\int_0^t e^{i (\rho_n(x) - \kappa_n(x))} \, dx  \longrightarrow  \int_0^t e^{i (\rho(x) - \kappa(x))} \, dx$$
as $n\to \infty$ for all $t \in [0, |\del P|]$.  In other words, {\em if the cap $\hat{P}_\mu$ exists}, then the desired boundary parametrization will be uniformly approximated by the curves $\hat{s}_n$ defined by 
	$$\hat{s}_n(t) = \int_0^t e^{i (\rho_n(x) - \kappa_n(x))} \, dx.$$
Note that the curves $\hat{s}_n$ are not necessarily closed loops, as the approximating polygonal caps $\hat{P}_{\mu_n}$ may not exist.  
\qed

\subsection{Circumference and curvature} \label{ss:circumferenceandcurvature}
If the boundary of the planar domain $P$ and the measure $\mu$ are smooth enough, then the curvature of \S\ref{curvature definition} satisfies the following relation, as a consequence of Theorem \ref{t:parametrization}.  

\begin{proposition} \label{p:smoothcurv}
Let $P$ be a planar shape with boundary parametrized by arclength by $s: [0,L] \to \del P$ such that $s$ is twice continuously-differentiable, and let $\mu$ be a probability measure on $\del P$ which is absolutely continuous with respect to arclength with a continuous density function.  Suppose the metric $\rho(P,\mu)$ exists.  For each $x \in \del P$, let $C(x,r)$ denote the circumference of a circle in $(\Chat, \rho(P,\mu))$ centered at $x$ of radius $r>0$.  Then
	$$\lim_{r\to 0^+} \frac{2\pi r - C(s(t),r)}{r^2} =  \delta(t),$$
 where $s^*\mu = \delta(t)\,dt$ on the interval $[0,L]$.  
\end{proposition}

It is interesting to compare the statement of Proposition \ref{p:smoothcurv} to the formula (\ref{point curvature}) for the Alexandrov curvature of a point, 
	$$\omega(\{x\}) = \lim_{r\to 0^+} \frac{2\pi r - C(x,r)}{r},$$
and to the Bertrand-Puiseux formula for the Gaussian curvature $\kappa$ when the metric on a surface is smooth, 
	$$\kappa(x) = \; \lim_{r\to 0^+}  \; 3 \, \frac{2\pi r - C(x,r)}{\pi r^3}$$
\cite[page 147]{Spivak:Vol2}.  In our setting, the curvature of the surface is supported on a 1-dimensional curve, so the circumference discrepancy is proportional to $r^2$.  

We begin with a simple geometric lemma.

\begin{lemma} \label{l:circles}
For real numbers $R>r>0$, let $A(R,r)$ be the arclength of the intersection of a closed disk of radius $R$ and a circle of radius $r$ centered at a boundary point of the disk.  Then $$\lim_{r \rightarrow 0^+} \frac{\pi r - A(R,r)}{r^2} = \frac{1}{R}.$$
\end{lemma}

\begin{proof}
Assume the center of the radius $r$ circle is at the origin in $\mathbb{R}^2$, and the disk of radius $R$ is tangent to the x-axis at the origin.  These two circles are given by the equations $x^2+(y-R)^2 = R^2$ and $x^2+y^2=r^2.$  These two circles intersect in two points: $ \left( \pm \sqrt{r^2-\frac{r^4}{4R^2}},\frac{r^2}{2R} \right)$. Hence, 
$A(R,r) = r \left(\pi - 2 \tan^{-1}\left(\frac{r^2}{\sqrt{4R^2r^2-r^4}}\right) \right).$
Then $$\lim_{r \rightarrow 0^+} \frac{\pi r - A(R,r)}{r^2} = \lim_{r \rightarrow 0^+}  \frac{ 2 \tan^{-1}\left(\frac{r^2}{\sqrt{4R^2r^2-r^4}}\right)}{r}= \frac{1}{R}.$$
\end{proof}

\proof[Proof of Proposition \ref{p:smoothcurv}.]
The curvature function of equation (\ref{curvaturefunction}) is computed as 
	
	$$\kappa(t) = \mu(s(0,t]) = \int_0^t \delta(x) \,dx.$$
For each $t\in [0,L]$ and each small $r>0$, the circumference $C(s(t),r)$ is the sum of the lengths of two circular arcs:  the arc in $P$ to the ``left" of $s(t)$ (relative to the counterclockwise orientation on $\partial P$), whose length we will denote by $C_r(t)$, and the arc in $\hat{P}_\mu$ to the ``right" of $\hat{s}(t)$ (relative to the clockwise orientation on $\partial \hat{P}_\mu$), whose length we will denote by $\hat{C}_r(t)$.   Classical plane geometry tells us that the radius of the osculating circle to the plane curve $s$ at $s(t)$ is $1/| s^{\prime \prime}(t) | = 1/|\alpha^{\prime}(t)|$, using the notation of Theorem \ref{t:parametrization}.  Likewise, from Theorem \ref{t:parametrization}, the radius of the osculating circle to the plane curve $\hat{s}$ at $\hat{s}(t)$ equals $1/ | \hat{s}^{\prime \prime}(t)| = 1/ | \alpha^{\prime}(t) - \kappa^{\prime}(t)|$. 

For $\alpha^{\prime}(t) > 0$, the osculating circle is to the left of $s(t)$, so 
 	$$\lim_{r \rightarrow 0} \frac{\pi r - C_r(t)}{r^2} = |\alpha^{\prime}(t)| = \alpha^{\prime}(t)$$ 
by Lemma \ref{l:circles}.  For $\alpha^{\prime}(t) < 0$, the osculating circle is to the right of $s(t)$, so 
	$$\lim_{ r\rightarrow 0} \frac{\pi r - C_r(t)}{r^2} 
		= \lim_{r \rightarrow 0} \frac{ \pi r - \left(2\pi r - A\left(\frac{1}{|\alpha^{\prime}(t)|},r\right)\right)}{r^2}
		 = - | \alpha^{\prime}(t)| = \alpha^{\prime}(t)$$
by Lemma \ref{l:circles}. Thus $\lim_{r \rightarrow 0} \frac{ \pi r - C_r(t)}{r^2} = \alpha^{\prime}(t)$, regardless of the sign of $\alpha'(t)$.  Similarly, 
$$\lim_{r \rightarrow 0} \frac{ \pi r - \hat{C}_r(t)}{r^2} = -(\alpha^{\prime}(t) - \kappa'(t)) = \delta(t) - \alpha^{\prime}(t)$$
regardless of the sign of $\alpha^{\prime}(t) - \kappa'(t)$.  Hence, 
$$ \lim_{r \rightarrow 0} \frac{2\pi r - C(s(t),r)}{r^2} =
  \lim_{r \rightarrow 0} \frac{ \pi r - C_r(t)}{r^2} + 
  \lim_{r \rightarrow 0} \frac{ \pi r - \hat{C}_r(t)}{r^2} = \alpha^{\prime}(t) + \delta(t) - \alpha^{\prime}(t) =\delta(t). $$
\qed

%%%%%%%%%

\bigskip
\section{Harmonic measure and holomorphic 1-forms}
\label{s:complex}

In this section, we present curvature in the setting of conformal metrics, allowing us to use tools from complex analysis to address our geometric questions.  This perspective was first formalized by Reshetnyak \cite{Reshetnyak}.  We present the proof of Theorem \ref{t:harmonic} and derive an alternative proof of the parametrization of the harmonic cap from Theorem \ref{t:parametrization}.  Finally, we revisit the general problem of existence of the metric $\rho(P,\mu)$ in Proposition \ref{p:existence}.  

\subsection{Complex-analytic point of view}
A  smooth conformal metric on a domain in $\C$ can be expressed as 
	$$\rho(z) |dz|$$
for a smooth and positive function $\rho$.  The metric has non-negative curvature if $U(z) = -\log \rho(z)$ is a subharmonic function.  Working with a more general class of metrics, we will only require that $U$ be subharmonic, not necessarily differentiable or everywhere finite.  We will also require that all pairs of points have finite distance from one another.   These requirements can be formulated in terms of the curvature of the metric, as we explain below.  

Formally, a conformal metric $\rho$ on $\Chat$ is a (singular) Hermitian metric on the tangent bundle $T\Chat \simeq \mathcal{O}_{\P^1}(2)$, and the curvature form of the metric is the positive measure given in local coordinates by 
	$$\omega_\rho = - \Delta \log \rho$$ 
(with $\Delta = 2i \del \delbar$ taken in the sense of distributions), so that 
	$$\int_{\Chat} \omega_\rho = 4\pi.$$  
In more classical terms, for a smooth metric $\rho$, the Gaussian curvature is computed locally as 
	$$\kappa_\rho = \frac{-\Delta \log \rho}{\rho^2}.$$
See, for example, \cite[\S 1.5]{Ahlfors:conformal} or \cite[\S 2.2]{Hubbard}.  

That $U = -\log\rho$ is subharmonic guarantees that the curvature form $\omega_\rho \geq 0$ as a distribution.  Finite diameter is guaranteed by the assumption that $\omega_\rho(\{z_0\}) < 2\pi$ for all $z_0 \in \Chat$ \cite[p.100]{Reshetnyak}.
 Recall from \S  \ref{ss:circumferenceandcurvature}
  that concentrated curvature, at points $z_0 \in\Chat$ where $0 < \omega_{\rho}(\{z_0\}) < 2\pi$, corresponds to cone points in the local geometry.  Also in this setting, a computation shows that the circumference $C(z_0, r)$ of a small circle around $z_0$ of radius $r>0$ will satisfy \cite[Lemma 8.1.1]{Reshetnyak}
  	$$\lim_{r \to 0^+} \frac{2\pi r - C(z_0,r)}{r} = \omega_{\rho}(\{z_0\}).$$

Conversely, every probability measure $\mu$ on $\Chat$ with $\mu(\{z\}) < 1/2$ for all $z$ gives rise to a conformal metric of finite diameter with curvature distribution $4\pi \mu$, unique up to scale.  Indeed, there is a one-to-one correspondence between probability measures $\mu$ on $\Chat$ and their potentials, up to an additive constant, which can be viewed as logarithmically-homogeneous, plurisubharmonic functions $G_\mu$ on the tautological line bundle $\C^2\setminus \{(0,0)\} \to \P^1$; see, e.g., \cite[Theorem 5.9]{Fornaess:Sibony} and \cite[Section 12]{D:lyap}.  The function $G_\mu$ will satisfy $(2\pi)^{-1} \Delta G_\mu(z,1) = \mu$ in local coordinates $z$ on $\Chat$, and the conformal metric is expressed as 
	$$\rho_\mu = e^{-2 G_\mu(z,1)} |dz|.$$
The identification between measures and their potentials is continuous, taking the $L^1_{loc}$ topology on potentials and the weak topology on measures.   Moreover, convergence of curvatures implies convergence of the metrics \cite[Theorem 7.3.1]{Reshetnyak}.

\subsection{Harmonic measure as curvature}  \label{harmonic metric}  Let $P$ be a compact, connected set in $\C$ containing at least 2 points, so that $P$ is a planar shape as defined in the Introduction.  Let $G_P: \C\to \R$ be the Green function for $P$; it is the unique continuous function on $\C$ satisfying (1) $G_P \equiv 0$ on $P$, (2) $G_P(z) = \log|z| + O(1)$ for $z$ near $\infty$, and (3) $G_P$ is harmonic on $\C\setminus P$.  Then define a metric on $\C$ by 
	$$\rho_P = e^{-2G_P(z)} |dz|.$$
By elementary potential theory, the function $G_P$ satisfies $G_P(z) = \log(z) + \gamma + o(1)$ for $z$ near $\infty$ for some real number $\gamma$, so the metric extends uniquely by continuity across $z=\infty$.  Note that this metric is {\em flat} (with 0 curvature) away from the boundary $\del P$.  
 Its curvature form $\omega_P = 2 \Delta G_P$ is equal to ($4\pi$ times) the harmonic measure on $\del P$ (more precisely, the harmonic measure for the domain $\Chat\setminus P$, relative to the point $\infty$).  
 
\begin{example}  \label{disk}  
Let $P$ be the closed unit disk.  Then $G_P(z) = \log^+|z| = \max\{0, \log|z|\}$, and the curvature form $\omega_P$ is arclength measure on the unit circle, normalized to have total length $4\pi$. By the symmetry of $P$, it is not hard to see that Alexandrov's realization of $(\Chat, \rho_P)$ will be the degenerate doubled flat disk.  
\end{example}

\subsection{The harmonic cap}  
Let $P$ be any planar shape.  Let $\Phi$ be the Riemann map from the complement of the unit disk to the complement of $P$, sending infinity to infinity.  Consider the holomorphic 1-form 	
	$$\eta  = \frac{1}{(\Phi^{-1}(z))^2} \; dz$$
on the complement of $P$.  Since the Green function satisfies
	$$G_P(z) = \log|\Phi^{-1}(z)|$$ 
on $\Chat\setminus P$, we see that $|\eta|$ is precisely the conformal metric $\rho_P$ defined above, when restricted to the complement of $P$.  Recall that Theorem \ref{t:harmonic} asserts that a Euclidean development of the harmonic cap of $P$ is given by the locally univalent function $g: \D\to\C$ defined by 
	$$g(z) = \int_0^z \Phi'(1/x) \, dx.$$
It also asserts that there exist examples where the locally univalent $g$ fails to be univalent.

\proof[Proof of Theorem \ref{t:harmonic}]
Define $F: \Chat\setminus P \to \C$ by 
	$$F(z) = \int_\infty^z \eta \; = \int_\infty^z \frac{1}{(\Phi^{-1}(\zeta))^2} \, d\zeta.$$
By definition, we have $\eta = dF = F^* (dw)$, where $dw$ is the standard holomorphic 1-form on the plane.  Since $|\eta|$ is the desired conformal metric, and since $|\eta| = F^* |dw|$, we conclude that $F$ is a Euclidean development of the harmonic cap parametrized by $z$ in $\Chat\setminus P$.  Now set $\iota(x) = 1/x$.  Then, to parameterize the cap by $z \in \D$, we pull $\eta$ back to $\D$ by $\Phi\circ \iota$, so that 
	$$D(z) = \int_0^z \iota^* \Phi^* \eta = \int_0^z \iota^* \left( \frac{\Phi'(\zeta)}{\zeta^2} \, d\zeta \right) =  - \int_0^z \Phi'(1/x) \, dx.$$
The local invertibility of $D$ is clear because $D'(z) = - \Phi'(1/x) \not= 0$ for all $x \in \D$.  Our desired function is $g(z) = -D(z)$, which is clearly an isometric presentation.  

It remains to observe that there exist planar shapes $P$ for which the development $g$ fails to be injective.  We constructed such an example in \S\ref{non-planar}, where $P$ is a square minus two thin spiral channels.  
\qed

\subsection{Harmonic cap boundary parametrization}
Here we present an alternative proof of the cap parametrization in Theorem \ref{t:parametrization}, in the special setting of harmonic measure.

As in Theorem \ref{t:parametrization}, assume that $P$ has a piecewise-differentiable boundary which is a Jordan curve parameterized by arclength by $s: [0,L] \to \C$.  Recall that $s^{\prime}(t) = e^{i \alpha(t)}$ for some piecewise continuous function $\alpha:[0,L] \rightarrow \R$.  Let $\Phi$ be a Riemann map from the complement of the unit disk to the complement of $P$, sending infinity to infinity.  Then $\Phi$ extends to a homeomorphism from the unit circle to the boundary of $P$. Define the conformal angle $\theta: [0,L] \rightarrow \R$ by 
	$$\theta(t) := \mathrm{arg}(\Phi^{-1}(s(t))).$$
Without loss of generality, we may assume $\theta(0)=0$ so that $\theta$ defines a homeomorphism from $[0,L]$ to $[0,2\pi]$.  It follows that the curvature function of (\ref{curvaturefunction}) for the harmonic measure $\mu$ on $\del P$ is equal to
	$$\kappa(t) = 4\pi \mu(s(0,t]) = 2\theta(t).$$
Therefore, from Theorem \ref{t:parametrization}, we know that the parametrization of the boundary of the harmonic cap is given by
\begin{equation}  \label{s hat theta}
   \hat{s}(t) = \int_0^t e^{i(\alpha(x) - 2\theta(x))} \, dx.
\end{equation}

Theorem \ref{t:harmonic} grants an alternate proof of (\ref{s hat theta}).  Indeed, with the $g: \D \to \C$ of Theorem \ref{t:harmonic}, a parametrization of the boundary of the harmonic cap is given by
 	$$\hat{s}(t) = -g(1/\Phi^{-1}(s(t))) = -g(e^{-i\theta(t)}).$$
Moreover, the derivative of $g$ is $g'(z) = \Phi'(1/z)$, and therefore,
\begin{eqnarray*}
\hat{s}'(t) &=& -g'(1/\Phi^{-1}(s(t)))\frac{-(\Phi^{-1})'(s(t)) \; s'(t)}{\Phi^{-1}(s(t))^2}   \\
	&=&  \frac{-\Phi'(\Phi^{-1}(s(t)))}{-\Phi'(\Phi^{-1}(s(t)))} \; \frac{s'(t)}{\Phi^{-1}(s(t))^2} \\
	&=& e^{i(\alpha(t) - 2\theta(t))}.
\end{eqnarray*}

\subsection{Metric existence for general measures} 
We conclude by returning to our original problem about the existence of a metric $\rho(P,\mu)$, for the case where $P$ is a planar shape with Jordan curve boundary and the probability measure $\mu$ is arbitrary.

Suppose that $J$ is a Jordan curve in $\Chat$, cutting the sphere into Jordan domains $A$ and $B$.  We may assume that $0 \in A$ and $\infty \in B$.  Suppose that $\nu$ is a probability measure supported on $J$, and let 
	$$U(z) = \int_\C \log|z-w| \, d\nu(w)$$
be a potential function for $\nu$ with logarithmic singularity at $\infty$.  The conformal metric 
	$$e^{-2U(z)}|dz|$$
on $\C$ extends to $\Chat$ and has curvature distribution equal to $4\pi\nu$.  Since $A$ is simply connected, there exists a non-vanishing analytic function $\phi: A \to \C$ so that 
	$$U(z) = \log|\phi(z)|.$$
The function $\phi$ is determined uniquely, up to postcomposition by a rotation.  Set
	$$f_\nu(z) = \int_0^z \frac{dz}{\phi(z)^2}$$
for $z \in A$.  Then $f_\nu: A \to \C$ is a locally-univalent Euclidean development of $A$ into the plane.  It extends continuously to the boundary curve $J$. This proves the following proposition.  

\begin{proposition}  \label{p:existence}
Let $P$ be a planar shape with Jordan curve boundary, and let $\mu$ be a probability measure supported on $\del P$.  The metric $\rho(P,\mu)$ on $\Chat$ exists if and only if there is a pair $(J, \nu)$ of a Jordan curve bounding a region $A$ in $\Chat$ and probability measure supported on $J$ so that $f_\nu(A) = P$ and $(f_\nu)_*\nu = \mu$.
\end{proposition}

When $\mu$ is the harmonic measure on $\del P$, observe that we may take $J = \del P$ and $\nu = \mu$ in the statement of Proposition \ref{p:existence}.  Indeed, the potential function for harmonic measure satisfies $U \equiv 0$ on $P$ so that $f_\nu = \mathrm{Id}$.  

\bigskip

%\bibliographystyle{amsalpha}
%\bibliographystyle{plain}
%\nocite{*}
%\bibliography{ConvexShapesBibliography}

\end{document}